\documentclass{article}
\usepackage[utf8]{inputenc}
\usepackage{mathtools}
\usepackage{amsthm}
\usepackage{amssymb}
\usepackage{mathrsfs}
\usepackage{geometry}
\usepackage[colorlinks,citecolor=blue]{hyperref}
\usepackage{todonotes}
\usepackage{cite}
\newtheorem{theorem}{Theorem}
\newtheorem{lemma}{Lemma}
\newtheorem{remark}{Remark}
\newtheorem{definition}{Definition}
\newtheorem{corollary}{Corollary}
\newtheorem{prop}{Proposition}
\numberwithin{equation}{section}

\title{First eigenvalue of embedded minimal surfaces in $S^3$  }  

\author{Yuhang Zhao\footnote{Y.Zhao:yuhangzhao@smail.nju.edu.cn}\\  $^{\ast \dagger }$Department of Mathematics,  Nanjing University, Nanjing 210093, P.R.China}
\date{}
\begin{document}

\maketitle
\begin{abstract}
We prove that for an embedded minimal surface $\Sigma$ in $S^3$, the first eigenvalue of the Laplacian operator $\lambda_1$ satisfies $\lambda_1\geq 1+\epsilon_g$, where $\epsilon_g>0$ is a constant depending only on the genus $g$ of $\Sigma$. This improves previous result of Choi-Wang \cite{ChW}.    
\end{abstract}

\section {Introduction.}

Let $X:\Sigma^n \rightarrow S^{n+1}$ be a minimal immersion, where $\Sigma$ is compact. It's well known that any coordinate function of $\mathbb{R}^{n+2}$ restricts to an eigenfunction of the Laplacian operator of $\Sigma$ with eigenvalue $n$. In particular, this implies that the first eigenvalue (of the Laplacian) of $\Sigma$ is smaller than or equal to $n$.

In \cite{Yau}, S.T.Yau raised the conjecture that {\em ``The first eigenvalue of any  compact  embedded minimal surfaces in $S^{n+1} $ is $n$''}. In \cite{ChW}, Choi and Wang made the first breakthrough.They used Reilly's formula to get $\lambda_1\geq \frac{n}{2}$. In fact,observed by Xu-Chen-Zhang -Chen \cite{Xu},  a  simple modification of their argument implies  $\lambda_1 >\frac{n}{2} $. Despite many results in special cases, for example \cite{ChS}\cite{TangYan}, Choi-Wang's theorem remains the best result concerning Yau's conjecture at present.

In this paper, we shall study the $n=2$ case of Yau's conjecture. Since the genus $g=0$ and $g=1$ embedded minimal surfaces are equator $S^2$ (Calabi \cite{Ca}) and the Clifford torus (Brendle \cite{Bre}) respectively, Yau's conjecture is true by direct computations in these cases. Moreover, Choe and Soret \cite{ChS} showed that $\lambda_1 = 2$ for embedded minimal surfaces in $S^3$ which are invariant under a finite group of reflections. In particular, Yau's conjecture also holds for the minimal surfaces constructed by Lawson \cite{Law}. For general minimal surfaces in $S^3$ with $g\geq 2$, not much is known.

In this paper, we shall prove:
\begin{theorem}[Main Theorem]\label{thm:main}
The first eigenvalue $\lambda_1$ of an embedded minimal surface $\Sigma$ in $S^3$ satisfies:
$$\lambda_1 \geq 1+\epsilon_g,$$ 
where $\epsilon_g>0$ is a constant depending only on the genus $g$.
\end{theorem}
 
Our proof starts in the same way as Choi and Wang \cite{ChW}, namely we solve the Dirichlet problem 
$$\Delta u=0 \ \text{in}\ \Omega,\quad u=f\ \text{on}\ \partial\Omega,$$
where $\Omega$ is one of the two domains bounded by $\Sigma$ and $f\in C^\infty(\Sigma)$ satisfies $\Delta_\Sigma f=-\lambda_1 f$. As Choi and Wang, we use the Reilly formula and integrate by parts. To improve, we need to carefully analyze the Hessian term $\int_\Omega |D^2 u|^2$, which is simply thrown aways in \cite{ChW}.  
 
One distinguishing property of embedding minimal surfaces among minimal immersions is that they alway have tubular neighborhoods, obtained from normal exponential maps. Let $N\Sigma$ be the normal bundle of $\Sigma$ in $S^3$, endowed with induced metric. In this case, $\Sigma$ divides $S^3$ into two connected components and  $N\Sigma$ is a trivial bundle. Fix one component and call it $\Omega$. The inner pointing unit normal vector field along $\partial\Omega=\Sigma$ is denoted by $\mathbf{n}$. Let $r_0$ be the supremum of all $r>0$ such that the normal exponential map $exp: N\Sigma\to S^3$ is a diffeomorphism on $\{(p, s\mathbf{n}(p))\in N\Sigma\ |\ s<r\}$\footnote{Its image is usually called a ``collar neighborhood of $\partial\Omega$".}. We call $r_0$ the {\em ``normal injectivity radius of $\Sigma$ in $\Omega$".} A key step in our proof is to get a uniform lower bound for $r_0$. In fact, we will prove that $r_0=\cot^{-1}\lambda_{max}$, where $\lambda_{max}$ is the maximum of the (positive) principal curvature. Since Choi and Schoen\cite{ChSch} proved that $\lambda_{max}$ has a uniform upper bound depending only on the genus $g$, $r_0$ has a uniform lower bound depending only on $g$. So we need only to bound $\lambda_1$ from below using $r_0$.

The second main ingredient of our proof is the use of Steklov eigenvalue for annulus regions within the collar neighborhood. We use the estimate of Colbois, Soufi, and Girouard \cite{CSG} for Steklov eigenvalues of $\Sigma\times [a, b]$ with product metric to bound the first Steklov eigenvalue of suitable  annulus regions in $\Omega$, from which our main theorem follows.Note that the idea of using the Dirichlet-to-Neumann map also appeared in \cite{Barros}.

In the next section, some preliminary calculations on minimal surfaces in $S^3$ using complex analytic charts are briefly reviewed. Then in section 3, we bound $r_0$. Though some results hold in the more general setting, we only prove the $S^3$ case, where everything can be written explicitly. Though we mainly concentrate in the minimal case, we prove the estimate of $r_0$ in the general mean convex setting. This result may be of some independent interest. Finally, we prove the main theorem in section 4 using Steklov eigenvalues.

{\bf Acknolwdgement:} The author is grateful to professor Yalong Shi for suggesting this problem and for helpful discussions.He also thanks professor Aldir Brasil for bringing \cite{Barros} to his sight.

\section{Preliminaries}

Let $X:\Sigma^2 \rightarrow S^3 $ be an  immersion. It's well known that 
the induced metric $X^{*} g_{S^3} $ determines uniquely a conformal structure.
If we choose a holomorphic coordinate system  $z =x+iy $ with respect to the conformal structure, the induced metric can be written as 
$$X^{*} g_{S^3} =F(dz\otimes d\overline{z} +d\overline{z}\otimes dz)=2F(dx\otimes dx+dy\otimes dy).$$
It is direct to check that we have the following identities:
\begin{equation}\label{eqn:basic}
\begin{aligned}
X \cdot X=1,\qquad X \cdot \frac{ \partial X }{\partial z  }&=0, \qquad X \cdot \frac{ \partial X }{\partial {\bar{z}}  }=0\\
\frac{\partial {X} }{ \partial {z} } \cdot \frac{\partial {X} }{ \partial z }=0,\quad \frac{ \partial {X} }{\partial {z}  } \cdot \frac{ \partial {X} }{\partial {\overline{z}}  }&=F,
\quad\frac{\partial^2X}{\partial {z}^2} \cdot \frac{\partial X}{\partial z}=0.
\end{aligned}
\end{equation}
We denote by $\bf n$ the unit normal vector field of $\Sigma$ and  use $H$  to represent the mean curvature of $\Sigma$ with respect to $\bf n$. Then the two principal curvatures of $\Sigma$ can be written as 
$H+\lambda ,H-\lambda$, where $\lambda\geq 0$ is a continuous  function on $\Sigma$ defined by $\lambda =\sqrt{H^2+1-K}$, where $K$ is the Gaussian curvature.

 At any point $q$, we can choose proper coordinate system, still denoted by $z=x+iy$, such that
\begin{equation}\label{eqn:derived}
\begin{aligned}
\mathbf{n}_x (q) &=(-H- \lambda X_x )(q),\quad \mathbf{n}_y (q)=(-H+\lambda X_y) (q),\\
X_{xx} (q) &=-2F(q)X(q)+\frac{1}{2F} (q)\frac{\partial {F} }{\partial x} (q) X_x (q)-\frac{1}{2F} (q) \frac{\partial {F} }{\partial y} (q) X_y (q)+2F(q)(H+\lambda ) (q) \mathbf{n}(q),\\
X_{yy} (q) &=-2F(q)X(q)+\frac{1}{2F} (q)\frac{\partial {F} }{\partial y} (q) X_y (q)-\frac{1}{2F} (q)\frac{\partial {F} }{\partial x} (q) X_x(q)-2F(q)(H-\lambda) (q) \mathbf{n}(q),\\
X_{xy} (q) &=\frac{1}{2F} (q)\frac{\partial {F} }{\partial y} (q) X_x (q)+\frac{1}{2F} (q) \frac{\partial {F} }{\partial x} (q) X_y (q),
\end{aligned}
\end{equation}
where the last three identities follows easily by differentiating \eqref{eqn:basic}.

By the Gauss-Codazzi equation, we have 
$$K=1+(H+\lambda)(H-\lambda).$$
Hence if $\Sigma$ is compact, by Gauss-Bonnet formula, we have 
$$\int_{\Sigma} (1+(H+\lambda)(H-\lambda) )d\sigma =4\pi (1-g),$$
where $g$ is the genus of $\Sigma$.
 
From this, we immediately get:
 \begin{lemma}\label{lem:lambda_max}
If $\Sigma$ is compact and $g\geq 1$, for immersion  $\Sigma^2 \rightarrow S^3$, there exits a point $p\in \Sigma $ such that $(H+\lambda)(p)\geq 1$ and $H(p)\geq 0$ or  $(H-\lambda )(p)\leq -1$ and $H(p)\leq 0$.
 \end{lemma} 
 \begin{proof}
If $g\geq 1$, then from the above formula, we get 
$$\int_{\Sigma} (1+(H+\lambda)(H-\lambda) )d\sigma \leq 0.$$
Then there exists a point $p$ such that $$(H+\lambda)(H-\lambda)(p) \leq -1.$$
If $H(p) \geq 0$,then because  $\lambda(q)\geq 0$, we have $$(H+\lambda)(p)\geq 1,\quad (\lambda-H)(p) >0.$$
The $H(p) \leq 0$ case is similar.
\end{proof}
 
\section{Cut point of $\Sigma$ in $S^3$ and its properties} 

In the following sections, the surface $\Sigma$ is always assumed to compact.
 If $X:\Sigma \rightarrow S^3 $ is an embedding, then we can identify $\Sigma $ with its image $X(\Sigma)$, and $X(\Sigma) $ divides $S^3$ into two bounded connected components: $\Omega_1,\Omega_2$.
We use $\mathbf n$  to represent the inner normal vector field of $\Sigma$ with respect to $\Omega_1 $, then the inner normal vector field of $\Sigma$ with respect to $\Omega_2$ is $-\mathbf n$. We use $\rho_{S^3} $ to denote the distance function of $S^3$.

For $\Omega_1$ or $\Omega_2$, we know that there exists $r'>0$ such that for any point $q\in \Sigma$, and for any point $\exp_{q} (d\mathbf{n}(q))$ or $\exp_{q} (-d\mathbf{n}(q)) (0\leq d\leq r')$,we have $\rho_{S^3} (\exp_{q} (d\mathbf{n}(q)),\Sigma)=\rho_{S^3} (exp_{q} (-d\mathbf{n}(q)),\Sigma)=d$.
 
\begin{definition}
For any point $q\in \Sigma$, let $r(q)=sup \{d\geq 0| \rho_{S^3}(\exp_{q} (d\mathbf{n}(q)),\Sigma)=d \}$. Then $r(q)$ is called the distance of the cut point of $q$ in $\Omega_1$.
 \end{definition}
 
 Since any geodesic in $S^3$ can return to its origin when it passes through the distance $2\pi$, we immediately have
 $\rho_{S^3}(\exp_{q} (r(q)\mathbf{n}(q)),\Sigma)=r(q) \in (0,\pi)$.
 By the  first variation formula of the distance function, we have : 
 \begin{lemma}
 $$\frac{ \cos r(q)  }{\sin r(q)} \geq  (H+ \lambda) (q) ,\quad\forall q \in \Sigma .$$ 
 \end{lemma}
   
   \begin{proof}
   
   By definition,we know the function $\langle\cos r(q) X(q)+\sin r(q) \mathbf{n}(q) ,X(x,y)\rangle $ achieves its maximum   at $q$,
where $\langle,\rangle$ represents the standard inner product in $\mathbb{R}^4$.
Hence $$\langle\cos r(q) X(q)+\sin r(q) \mathbf{n}(q) ,X_{xx} (q)\rangle\leq 0.$$
  By $(2.1)$,we have $$ -\cos r(q) +(H+\lambda) (q) \sin r(q) \leq  0 ,$$ 
  that is,$$\frac{ \cos r(q)  }{\sin r(q)} \geq  (H+ \lambda) (q) .$$ 
   \end{proof}
  
In fact, the point $\exp_{q} (\cot^{-1} ((H+\lambda) (q) )\mathbf{n}(q) ) $ is the first focal point of $\Sigma$ in the usual sense.
 
 The following proposition can be proven by modifying standard arguments in Riemannian geometry. Hence we omit the details.
 
 \begin{prop}\label{prop:cut}
  $r(q)$ is  a continuous  function defined on $\Sigma $,hence it can achieve its maximum and minimum.And for any point $q\in \Sigma$,if the point $\exp_{q} (r(q) \mathbf{n}(q))$ is not the first focal point,then there exists another point $q'\neq q$ such that 
  $$r(q')=r(q),\exp_{q} (r(q) \mathbf{n}(q))=\exp_{q'} (r(q') \mathbf{n}(q')).$$
\end{prop} 

Let $$r_0=min_{q\in \Sigma} r(q).$$
Clearly,$0<r_0\leq  \cot^{-1} ((H+\lambda)_{max} ) <\pi$. And we call $r_0$ the {\em normal injectivity radius of $\Sigma$ with respect to  $\Omega_1$}.

With these preparations, now we have
\begin{theorem}\label{thm:radius}
For  mean convex  embedding $X :\Sigma\rightarrow S^3$ (i.e. $H \geq 0$), we have 
$$r_0=\cot^{-1} ((H+\lambda)_{max} )\leq \frac{\pi}{2}.$$ Moreover, if the genus $g$ of $\Sigma$ satisfies $g\geq 1$, we have $r_0\leq \frac{\pi}{4}$.
\end{theorem}

\begin{proof}
We prove the theorem by contradiction.

First, let $$\Sigma_{r_0} =\{q'\in \Omega_1 | \rho_{S^3} (q',\Sigma)=r_0 \}=\{\exp_{q} (r_0 \mathbf{n}(q))=\cos r_0 X(q) +\sin r_0 \mathbf{n}(q)| q\in\Sigma \}.$$
Suppose $r_0<\cot^{-1} ( (H+\lambda)_{max} )$, by Proposition \ref{prop:cut}, we can find two points $q_1\neq q_2\in \Sigma $ such that $$r(q_1)=r(q_2)=r_0,\quad \exp_{q_1} (r_0 \mathbf{n}(q_1))=\exp_{q_2} (r_0 \mathbf{n}(q_2)).$$

We can choose coordinate systems of $q_1,q_2$ :$z_1=x_1+i y_1,z_2=x_2+iy_2$ satisfying \eqref{eqn:derived}, and corresponding metric coefficients are denoted by $F,F'$.

Since $$\rho_{S^3}  (q_1,\Sigma_{r_0} )=\rho_{S^3}  (q_2,\Sigma_{r_0})  =r_0
=\rho_{S^3} (q_1,\exp_{q_1} (r_0 \mathbf{n}(q_1)))=\rho_{S^3} (q_2,\exp_{q_1} (r_0 \mathbf{n}(q_1))),$$
by triangle inequality,
 $$\rho_{S^3} ( \cos r X(q_2)+\sin r \mathbf{n}(q_2),\Sigma_{r_0} )=r_0-r,\quad \forall r\in [0,r_0].$$
 That is, the function $\langle\cos r X(q_2)+\sin r \mathbf{n}(q_2),\cos r_0 X(x_1,y_1)+ \sin r_0 \mathbf{n}(x_1,y_1)\rangle$ achieves its maximum at $q_1$, $\forall r\in [0,r_0]$.

Hence by \eqref{eqn:derived},
\begin{align*}
\langle\cos r X(q_2)+\sin r \mathbf{n}(q_2),(\cos r_0 -\sin r_0 (H+\lambda )(q_1)) X_{x_1} (q_1) \rangle&=0,\\
\langle\cos r X(q_2)+\sin r \mathbf{n}(q_2), (\cos r_0 -\sin r_0 (H-\lambda )(q_1))X_{y_1} (q_1) \rangle &=0,\quad \forall r\in [0, r_0].
\end{align*}
Since $\frac{\cos r_0}{\sin r_0} >(H+\lambda) (q_1) $,
we have  $$\cos r X(q_2)+\sin r \mathbf{n}(q_2) \perp X_{x_1} (q_1),X_{y_1} (q_1) ,\quad\forall r\in [0, r_0].$$
Thus 
$$-\sin r_0 X(q_1)+\cos r_0 \mathbf{n}(q_1)=-(-\sin r_0 X(q_2)+\cos r_0 \mathbf{n}(q_2) )\quad  (q_1\neq q_2).$$
Hence there exists a  geodesic $\gamma(r)=\exp_{q_1} (r \mathbf{n}(q_1) ) $ with $\gamma(2r_0)=q_2$.

Notice that for very small $\delta>0$, we have  $\exp_{q_1} (d\mathbf{n}(q_1))\in \Omega_2$ when $2r_0 < d\leq 2r_0 +\delta$ and that
 $$\rho_{S^3} (\exp_{q_1} (d\mathbf{n}(q_1)),\Sigma_{r_0} )=|r-r_0|, \quad \forall d\in [0, 2r_0 +\delta].$$

We can assume $(H+\lambda ) (q_1)\geq (-H+\lambda)(q_2)$, for otherwise we will get $(H+\lambda)(q_1)<(-H+\lambda)(q_2)$, and from the mean convexity assumption, we have
$$(H+\lambda)(q_2)\geq (-H+\lambda)(q_2)>(H+\lambda)(q_1)\geq (-H+\lambda) (q_1).$$
Then we can simply switch $q_1$ and $q_2$.

In a coordinate neighborhoods of $q_1$  and $q_2$, we can locally define two functions respectively :
\begin{align*}
g_d (x_1,y_1)&= \langle\cos d X(q_1)+\sin d  \mathbf{n}(q_1) ,\cos {r_0} X(x_1,y_1)+\sin {r_0} n(x_1,y_1) \rangle,\\
h_d (x_2,y_2)&= \langle\cos d X(q_1)+\sin d  \mathbf{n}(q_1) ,\cos {r_0} X(x_2,y_2)+\sin {r_0} n(x_2,y_2) \rangle,\quad d \geq 0
\end{align*}
Then when $0\leq d \leq 2 r_0+\delta $, $g_d $ achieves its maximum at $(0,0)$ (that is, $q_1$).
By \eqref{eqn:derived}, when $0<d \leq 2 r_0+\delta $
$$ 0\geq (g_d)_ {x_1 x_1} (q_1)=2F(q_1) (-\cos d +(H+\lambda )(q_1) \sin d )(\cos r_0 -(H+\lambda )(q_1) \sin r_0 ),$$
that is,
$$ \frac{\cos d}{ \sin d} \geq (H+\lambda) (q_1) ,\quad 0<d \leq 2 r_0 +\delta,\quad r_0<\frac{\pi}{4}.$$

Since $(H+\lambda )(q_1) )\geq (-H+\lambda)  (q_2)$, we can choose $2r_0 <r'<\frac{\pi}{2} +\delta'$ $(\delta'>0$ is small) such that 
$$r'>\cot^{-1} ((H+\lambda)(q_1)),\quad 2r_0+\cot^{-1} ((-H+\lambda) (q_2) )>r'.$$
Hence by \eqref{eqn:derived} again,
\begin{align*}
(g_{r'})_{x_1 x_1} (q_1)&= 2F(q_1) (-\cos r' +(H+\lambda) (q_1) \sin r' )(\cos r_0 -(H+\lambda) (q_1) \sin r_0 )>0,\\
(h_{r'} )_{x_2} (q_2)&=(h_{r'} )_{y_2} (q_2) =0=( h_{r'} )_{x_2 y_2} (q_2),\\
 (h_{r'} )_{x_2 x_2} (q_2)&=2F'(q_2)\langle\cos r' X(q_1)+\sin r' \mathbf{n}(q_1),-X(q_2)+(H+\lambda) (q_2) \mathbf{n}(q_2) \rangle (\cos r_0 -\lambda (q_2) \sin r_0 )<0,\\
 (h_{r'} )_{y_2 y_2} (q_2)&=2F'(q_2)\langle\cos r' X(q_1)+\sin r' \mathbf{n}(q_1),-X(q_2)+(H-\lambda)  (q_2) \mathbf{n}(q_2) \rangle (\cos r_0 -\lambda (q_2) \sin r_0 )<0,
\end{align*}
where  we use the condition: $\cos r' X(q_1)+\sin r'  \mathbf{n}(q_1) =\cos (r'-2r_0) X(q_2)-\sin (r'-2r_0) \mathbf{n}(q_2)$ and $\frac{\cos (r'-2r_0)}{\sin (r'-2r_0)}>(-H+\lambda)(q_2)$.
  
 Hence there exists a point $q_1'$ near $q_1$ such that  
 \begin{align*}
 &\quad \rho_{S^3} (\cos r' X(q_1)+\sin r'  \mathbf{n}(q_1),\cos r_0 X(q_1')+\sin r_0 \mathbf{n}(q_1') )\\
 &<r'-r_0=\rho_{S^3} (\cos r' X(q_1)+\sin r'  \mathbf{n}(q_1),\cos r_0 X(q_1)+\sin r_0 \mathbf{n}(q_1) ),
 \end{align*}
 and for any point $q\neq q_2$ near $q_2$,
\begin{align*}
&\quad \rho_{S^3} \cos r' X(q_1)+\sin r'  \mathbf{n}(q_1),\cos r_0 X(q)+\sin r_0 \mathbf{n}(q) )\\
&>r'-r_0=\rho_{S^3}(\cos r' X(q_1)+\sin r'  \mathbf{n}(q_1),\cos r_0 X(q_0)+\sin r_0 \mathbf{n}(q_0) ).
\end{align*} 

We claim that from these we can arrive at a contradiction.

In fact, we can choose a small neighborhood $U\subseteq \Sigma $ of $q_2$ and $\epsilon>0$ such that: 
$$\gamma (x_2,y_2,d) =\cos d\ X(x_2,y_2)+\sin d\ n(x_2,y_2):U\times (r_0-\epsilon,r_0+\epsilon )\rightarrow  \gamma (U\times (r_0-\epsilon,r_0+\epsilon ) )\subseteq \Omega_1 $$ is  a  diffeomorphism.

When we choose $q_1'$ close to  $q_1$ such that $\cos r_0 X(q_1')+\sin r_0 \mathbf{n}(q_1') \in  \gamma (U\times (r_0-\epsilon,r_0+\epsilon ) )$,  clearly, the point $\cos r_0 X(q_1')+\sin r_0 \mathbf{n}(q_1') \in \gamma (U\times [r_0,r_0+\epsilon ) ) $ by the definition of cut points.

Since $q_1'$ is close to $q_1$, in the geodesic joining $\cos r' X(q_1)+\sin r' \mathbf{n}(q_1) $ and $\cos r_0 X(q_1')+\sin r_0 \mathbf{n}(q_1')$, there exists a point $a \in \gamma (U\times (r_0-\epsilon,r_0 ) ) $ and the segment joining $a$ and $\cos r_0 X(q_1')+\sin r_0 \mathbf{n}(q_1')$ is contained in $\gamma (U\times (r_0-\epsilon,r_0+\epsilon ) )$.

We can clearly see that the segment joining $a$ and $\cos r_0 X(q_1')+\sin r_0 \mathbf{n}(q_1')$ must intersect the set  $\gamma (U\times r_0 )$ at a point $a' \neq \cos r_0 X(q_1)+\sin r_0 \mathbf{n}(q_1)$.

That is to say, the time when the geodesic joining $\cos r' X(q_1)+\sin r' \mathbf{n}(q_1) $ and $\cos r_0 X(q_1')+\sin r_0 \mathbf{n}(q_1')$  passes through $a'$ is earlier  than or equal to the time it passes through $\cos r_0 X(q_1')+\sin r_0 \mathbf{n}(q_1')$.

However, from the above, we know that 
$$\rho_{S^3}  (\cos r' X(q_1)+\sin r'  \mathbf{n}(q_1),\cos r_0 X(q_1')+\sin r_0 \mathbf{n}(q_1') )<r'-r_0 $$
and 
$$\rho_{S^3} (\cos r' X(q_1)+\sin r'  \mathbf{n}(q_1),\cos r_0 X(q_1)+\sin r_0 \mathbf{n}(q) ) > r'-r_0 ,$$
so we get a contradiction.

The rest of the theorem can be obtained  from Lemma \ref{lem:lambda_max}.
\end{proof}

\begin{remark}
The same proof works for closed and strictly mean convex hypersurfaces in $\mathbb{R}^3$.
\end{remark}

For embedded minimal surface $\Sigma$, both regions $\Omega_1,\Omega_2$ have mean convex boundaries, hence the normal injectivity radii with respect to $\Omega_1$ and $\Omega_2$ are same.And from Theorem \ref{thm:radius} we get:

\begin{corollary}\label{coro:r_0}
For minimal embedding $X:\Sigma\rightarrow S^3$, $r_0$ has a lower bound depending only on the genus $g$  of $\Sigma$. Moreover,if $g=0$, $r_0=\frac{\pi}{2}$; if $g=1$, $r_0=\frac{\pi}{4}$; and if $g>1$, $r_0<\frac{\pi}{4}$.
\end{corollary} 

\begin{proof}
For minimal embedding, we get $r_0=\cot^{-1} \lambda_{max}$. On the other hand, by \cite {ChSch}, we know that $\lambda_{max}$ has a uniform upper bound depending only on $g$. So $r_0$ has  a uniform lower bound depending only on $g$.
For $g=0,1$, since the embedding is explicit by \cite{Ca} and \cite{Bre}, we get the conclusion by direct computations.
For $g>1$, we must have $\lambda_{max}>1$ and hence $r_0=\cot^{-1} \lambda_{max}<\frac{\pi}{4}$.
\end{proof}

\section{Proof of the Main Theorem}
														
First, we review briefly the argument of Choi-Wang\cite{ChW}. We solve the Dirichlet problem in $\Omega=\Omega_1$ or $\Omega_2$:
$$\triangle u=0,\quad u | _{\Sigma } =f,$$
where $\triangle_{\Sigma} f=-\lambda_1 f,f\neq 0 $.

By Bochner's formula,
$$ \triangle |\nabla u |^2=2 | D^2 u |^2 +4| \nabla u|^2.$$
Using integration by parts, we obtain:
$$\int_{\Sigma} g_{S^3} (D_{v} \nabla u,\nabla u) =\int_{\Omega} | D^2 u |^2 +2 \int_{\Omega} | \nabla u|^2,$$
where $v$ is the outer normal vector of $\Omega$.
Since $$g_{S^3} (D_{v} \nabla u,\nabla u) =u_{v}u_{vv}+\nabla f\cdot \nabla u_{v}-h_{v}(\nabla f,\nabla f),u_{vv}=-\triangle_{\Sigma} f=\lambda_{1} f,$$
where $h_{v} $ is the second fundamental form with respect to $v$.
Using integration by parts again:
$$ \lambda_1 \int_{\Omega}  | \nabla u |^2 =\frac{1}{2} \int_{\Omega} | D^2 u |^2+\int_{\Omega} | \nabla u|^2 +\frac{1}{2}
\int_{\Sigma} h_{v} (\nabla f,\nabla f).$$

Since $$\int_{\Sigma} h_{v} (\nabla f,\nabla f)=-\int_{\Sigma} h_{-v} (\nabla f,\nabla f),$$
we can choose $\Omega_{1}$ or $\Omega_2$ such that 
$$ \int_{\Sigma} h_{v} (\nabla f,\nabla f) \geq 0 \ or \ \int_{\Sigma} h_{-v} (\nabla f,\nabla f)\geq 0.$$
And we denote this region by $\Omega$. Hence 
\begin{equation}\label{eqn:ChW}
\lambda_1 \geq 1+ \frac{\int_{\Omega} | D^2u |^2 }{2\int_{\Omega}   | \nabla  u  |^2     } .
\end{equation}
This is the starting point of our proof.

\begin{proof}[Proof of the Main Theorem]
By\cite{SchYau} ,we know 
$$ | D^2 u |^2 \geq \frac{3}{2} | \nabla|\nabla u| |^2 .$$
Hence by Bochner's formula again,
$$| \nabla u |   \triangle | \nabla u |  \geq \frac{1}{2} | \nabla|\nabla u| |^2+2 | \nabla  u  |^2 .   $$
We choose a cut off function $\eta$ with $0\leq \eta \leq 1$ and $\eta|_{\Sigma} =0 $ ,then 
$$\int_{\Omega} | \nabla|\nabla u| |^2\geq \int_{\Omega} \eta^2 | \nabla|\nabla u| |^2=-\int_{\Omega} \eta^2 | \nabla u | \triangle | \nabla u | -2\int_{\Omega} \eta | \nabla u | \nabla \eta \cdot \nabla | \nabla u |.$$
Thus
$$\int_{\Omega} \eta^2 | \nabla|\nabla u| |^2 \leq -\frac{1}{2} \int_{\Omega} \eta^2 | \nabla|\nabla u| |^2- 2\int_{\Omega}\eta^2 | \nabla u | ^2+2\sqrt{ \int_{\Omega}\eta^2 | \nabla | \nabla u | | ^2\int_ {\Omega}  | \nabla\eta |^2 | \nabla u | ^2 },$$
and hence
$$\sqrt{\int_{\Omega} \eta^2 | \nabla|\nabla u| |^2 } \geq \frac{2}{3} (\sqrt{ \int_{\Omega} | \nabla\eta |^2 | \nabla u | ^2 }-\sqrt{  \int_{\Omega} | \nabla\eta |^2 | \nabla u | ^2-3 \int_{\Omega}\eta^2 | \nabla u | ^2} ).$$
Equivalently,
$$\int_{\Omega} \eta^2 | \nabla|\nabla u| |^2  \geq \frac{ ( \int_{\Omega}\eta^2 | \nabla u | ^2)^2  } { \int _{\Omega}                  | \nabla \eta |^2 | \nabla u |^2  }.$$
Combining with \eqref{eqn:ChW}, we get
\begin{equation}\label{eqn:lambda_1}
    \lambda_1 \geq 1+\frac{3}{4} \frac{ ( \int_{\Omega}\eta^2 | \nabla u | ^2)^2           }{ \int _{\Omega}                  | \nabla \eta |^2 | \nabla u |^2  \int_{\Omega}   | \nabla  u  |^2  }   .
    \end{equation}

In the following, we use $l$ to denote the distance to $\Sigma$. We choose $\eta$ such that $\eta=\frac{l}{ \epsilon r_0 } (0<\epsilon <1)$ for $ 0\leq l\leq \epsilon r_0 $ and equals to 1 elsewhere. Also choose $\theta$ such that $\theta =\frac{l}{ \rho  r_0 } (0<\rho <1)$ for $ 0\leq l\leq \rho r_0 $ and equal to 1 elsewhere.

From now on, we always assume that the genus $g>1$. Then by Corollary \ref{coro:r_0}, we only need to work under the extra condition $r_0<\frac{\pi}{4}$.

Since the Dirichlet integral of harmonic functions are minimal among functions with fixed boundary values, and the volume element of the region $\{0\leq  l \leq \epsilon r_0\} $ can be represented by $(\cos^2 l-\lambda^2 \sin^2 l ) d\sigma$ ($d\sigma$ is the volume element of $\Sigma$),  we have

$$\begin{aligned}
\int_{\Omega}   | \nabla  u  |^2 &\leq \int_{l\leq \epsilon r_0} | \nabla ((1-\theta ) f) |^2 \\
&= \int_{0} ^{\rho  r_0} dl \int_{\Sigma} [\frac{1}{\rho^2 r_0 ^2} f^2+(1-\frac{l}{\rho r_0 })^2 | \nabla f |^2 ] (\cos^2 l-\lambda^2 \sin ^2 l) d\sigma\\
&=\frac{1}{\rho^2 {r_0}^2}  \int_{0} ^{\rho r_0} dl \int_{\Sigma}  f^2(\cos^2 l-\lambda^2 \sin ^2 l) d\sigma\\
&\quad + \int_{0} ^{\rho r_0} dl \int_{\Sigma} (1-\frac{l}{\rho r_0 })^2 [ \frac{\cos l+\lambda \sin l  }{\cos l-\lambda \sin l} \frac{1}{2F} (\frac{\partial f}{\partial x})^2+\frac{\cos l-\lambda \sin l}{\cos l+\lambda \sin l} \frac{1}{2F} (\frac{\partial f}{\partial y})^2 ]  d\sigma\\
&\leq \frac{1}{\rho^2 {r_0}^2}  \int_{0} ^{\rho r_0}  \cos ^2 l dl
\int_{\Sigma} f^2 d\sigma +  \int_{0} ^{\rho r_0} (1-\frac{l}{\rho r_0 })^2\frac{\sin (r_0+l)}{\sin (r_0-l) } dl \int _{\Sigma}   | \nabla_{\Sigma} f |^2   d\sigma \\
&\leq  [\frac{1}{\rho r_0 } +\frac{\rho (1+\rho) }{3(1-\rho)} r_0 \lambda_1 ] \int_{\Sigma} f^2 d\sigma,
\end{aligned}$$
where we view $f$ as a function in  the region $\{0\leq  l\leq \rho r_0\} $ and use $\frac{\cos r_0}{\sin r_0} \geq \lambda $ and $\triangle_{\Sigma} f=-\lambda_1 f $.
Then 
$$\int _{\Omega} | \nabla \eta |^2 | \nabla u |^2\leq \frac{1}{\epsilon^2 r_0 ^2} \int_{\Omega}   | \nabla  u  |^2
\leq \frac{1}{\epsilon^2 r_0 ^2}  [\frac{1}{\rho r_0 } +\frac{\rho (1+\rho) }{3(1-\rho)} r_0 \lambda_1 ] \int_{\Sigma} f^2 d\sigma.$$
Since 
$$\int_{\Omega}\eta^2 | \nabla u | ^2 \geq \int_{l\geq \epsilon r_0} | \nabla u | ^2,$$
we can plug the above results into \eqref{eqn:lambda_1} to get
\begin{equation}\label{eqn:lambda1_final}
    \lambda_1 \geq 1+\frac{3}{4} [\frac{\epsilon \rho {r_0}^2}{1 +\frac{\rho^2 (1+\rho) }{3(1-\rho)} {r_0}^2 \lambda_1 } \frac{\int_{l\geq \epsilon r_0} | \nabla u | ^2}{\int_{\Sigma} f^2 d\sigma }]^2 .
    \end{equation}

In the following,we need to estimate the term $\frac{\int_{l\geq \epsilon r_0} | \nabla u | ^2}{\int_{\Sigma} f^2 d\sigma }$.

We need the following lemma on the first nonzero Steklov eigenvalue:
\begin{lemma}
Under the assumptions of the main theorem, the first nonzero Steklov eigenvalue $\mu_1$ of the region $\{l\geq \epsilon r_0\} $ satisfies:
\begin{equation}\label{eqn:Steklov_final}
\mu_1\geq  \sup_{\epsilon'\in (\epsilon, 1)}\frac{1-\epsilon'}{1+\epsilon '} \sqrt{\lambda_1}
 \tanh (\frac{ \sqrt{\lambda_1 } (\epsilon'-\epsilon) r_0 }{2} ).
 \end{equation}
\end{lemma}

\begin{proof}
First we solve the Dirichlet problem in $\{l\geq \epsilon r_0\}$ :
$$\Delta v=0,\quad v|_{l=\epsilon r_0} =g,$$
where $g$ is the eigenfunction of the Dirichlet to Neumann operator with respect to $\mu_1$, that is $-v_l=\mu_1 g,  g\neq 0$.

For any $\epsilon<\epsilon'<1$, we have
 \begin{equation}\label{eqn:Steklov}
\begin{aligned}
\mu_1 & =\frac{\int_{l\geq \epsilon r_0}  | \nabla v  | ^2}{ \int_{l=\epsilon r_0} g^2 d\sigma_{\epsilon r_0} } \geq \frac{\int_{ \epsilon r_0 \leq l \leq \epsilon' r_0}  | \nabla v  | ^2}{ \int_{l=\epsilon r_0} g^2 d\sigma_{\epsilon r_0} }\\
&=\frac{1}{\int_{l=\epsilon r_0} g^2 d\sigma_{\epsilon r_0}} 
\int_{ \epsilon r_0 } ^{\epsilon' r_0}  dl \int_{\Sigma}   [(\frac{\partial v}{\partial l})^2 (\cos^2 l-\lambda^2 \sin ^2 l )  \\
&\quad +\frac{\cos l+\lambda \sin l  }{\cos l-\lambda \sin l} \frac{1}{2F} (\frac{\partial v}{\partial x})^2+\frac{\cos l-\lambda \sin l}{\cos l+\lambda \sin l} \frac{1}{2F} (\frac{\partial v}{\partial y})^2 ]  d\sigma   \\
&\geq \frac{1}{\int_{l=\epsilon r_0} g^2 d\sigma_{\epsilon r_0}} \int_{ \epsilon r_0 }^{\epsilon' r_0}  dl \int_{\Sigma}   [\frac{\sin (r_0 +l)\sin (r_0 - l)}{\sin^2 r_0} (\frac{\partial v}{\partial l})^2 \\
&\quad +\frac{\sin(r_0-l ) }{\sin (r_0+l) }  | \nabla_{\Sigma} v |^2 ]  d\sigma \\
 &\geq \frac{1-\epsilon'}{1+\epsilon'} \frac{\int_{ \epsilon r_0 }^{\epsilon' r_0}  dl \int_{\Sigma}   [ (\frac{\partial v}{\partial l})^2 + | \nabla_{\Sigma } v |^2 ]  d\sigma  } { \int_{l=\epsilon r_0} g^2 d\sigma_{\epsilon r_0} },
 \end{aligned}
 \end{equation}
where we use the fact
$$\frac{\cos r_0}{\sin r_0} \geq \lambda,\quad \frac{\sin (r_0 +l)\sin (r_0 - l)}{\sin^2 r_0} >\frac{\sin(r_0-l ) }{\sin (r_0+l) } >\frac{1-\epsilon'}{1+\epsilon'},$$
  and $d\sigma_{s} =(\cos^2 s-\lambda^2\sin^2 s) d\sigma (0\leq s <r_0) $.
 
If we consider the cylinder $[\epsilon r_0,\epsilon' r_0]\times \Sigma $ with  product metric,then we can choose a constant $c $ such that 
$$\int_{\Sigma} (g-c)d\sigma +\int_{\Sigma} (v|_{l=\epsilon'r_0} -c) d\sigma =0.$$
Since 
$$\int_{l=\epsilon r_0} g d\sigma_{\epsilon r_0} =0,$$
we have
$$\int_{l=\epsilon r_0} (g-c) ^2d\sigma_{\epsilon r_0 }  +\int_{l=\epsilon'r_0} (v -c)^2 d\sigma_{\epsilon' r_0} \geq  \int_{l=\epsilon r_0} g^2 d\sigma_{\epsilon r_0},$$
and 
$$\int_{l=\epsilon r_0} (g-c) ^2d\sigma_{\epsilon r_0 }  +\int_{l=\epsilon'r_0} (v -c)^2 d\sigma_{\epsilon' r_0}  \leq \int_{\Sigma } (g-c) ^2d\sigma  +\int_{\Sigma } (v|_{l=\epsilon'r_0} -c)^2 d\sigma.$$
Plugging into \eqref{eqn:Steklov} and use the definition of the first Steklov eigenvalue:
 $$\mu_1\geq \frac{1-\epsilon'}{1+\epsilon '} \mu_1 |_{ [\epsilon r_0,\epsilon' r_0]\times \Sigma } .$$          
By \cite{CSG} and $r_0<\frac{\pi}{4} $,we know the first Steklov eigenvalue of $ [\epsilon r_0,\epsilon' r_0]\times \Sigma $ is 
$$\sqrt{\lambda_1}
\tanh (\frac{ \sqrt{\lambda_1 } (\epsilon'-\epsilon ) r_0      }{2} ).$$
So we get
$$\mu_1 \geq \frac{1-\epsilon'}{1+\epsilon '} \sqrt{\lambda_1}
\tanh (\frac{ \sqrt{\lambda_1 } (\epsilon'-\epsilon) r_0) }{2} ).$$

\end{proof}

Now we come back to the proof of main theorem. Combining the above lemma with \eqref{eqn:lambda1_final}, we get for any $\epsilon'\in (\epsilon,1)$:

\begin{equation}\label{eqn:lambda-mu}
    \lambda_1 \geq 1+\frac{3}{4} [\frac{\epsilon \rho {r_0}^2}{1 +\frac{\rho^2 (1+\rho) }{3(1-\rho)} {r_0}^2 \lambda_1 } 
 \frac{1-\epsilon'}{1+\epsilon '} \sqrt{\lambda_1}
\tanh (\frac{ \sqrt{\lambda_1 } (\epsilon'-\epsilon) r_0) }{2} ) 
\frac{\int_{l= \epsilon r_0} (u-\phi (\epsilon r_0) )^2 d\sigma_{\epsilon r_0}   }{\int_{\Sigma} f^2 d\sigma }]^2,
\end{equation}
where $\phi(s) $ is a function such that 
$$\int_{l= s} (u-\phi (s ) ) d\sigma_{s} =0,\quad \phi (0)=0.$$
We need to estimate the term
$$\frac{\int_{l= \epsilon r_0} (u-\phi (\epsilon r_0 ) )^2 d\sigma_{\epsilon r_0}   }{\int_{\Sigma} f^2 d\sigma  } .$$

From the above estimate on $\int_{\Omega} | \nabla u |^2$ and $\frac{\cos r_0}{\sin r_0}\geq \lambda$, $\forall 0\leq s\leq \epsilon r_0$, we have
$$\begin{aligned} 
\frac{\partial }{\partial s}  \int_{l=  s} (u -\phi (s) )^2 d\sigma_{s} & =2\int_{r=  s} (u-\phi(s))( u_l -\phi'(s) ) d\sigma_{s} \\
& \quad-\sin (2s) \int_{\Sigma} (1+\lambda^2) (u|_{r=s}-\phi(s))^2 d\sigma \\
&\geq-2\int_{\Omega} | \nabla u |^2 -\sin (2s) \int_{\Sigma} (1+\lambda^2) (u|_{l=s}-\phi(s) )^2 d\sigma\\ 
&\geq - 2[\frac{1}{\rho r_0 } +\frac{\rho(1+\rho) }{3(1-\rho)} r_0 \lambda_1 ] \int_{\Sigma} f^2 d\sigma \\
&\quad-\frac{\sin(2s) }{\sin (r_0+s) \sin (r_0-s) }  \int_{l=s} (u -\phi (s) )^2 d\sigma_{s}  \\
&\geq - 2[\frac{1}{\rho r_0 } +\frac{\rho (1+\rho) }{3(1-\rho)} r_0 \lambda_1 ] \int_{\Sigma} f^2 d\sigma \\
&\quad-\frac{2\sin(\epsilon r_0 )}{\sin (1+\epsilon) r_0 \sin (1-\epsilon )r_0} \int_{l=s} (u -\phi (s) )^2 d\sigma_{s} .
\end{aligned} $$

Hence 
$$\begin{aligned}
&\frac{\partial }{\partial s}  [\exp ( {\frac{2\sin(\epsilon r_0 ) s}{\sin (1+\epsilon) r_0 \sin (1-\epsilon )r_0 } }) \int_{l=  s} (u -\phi (s) )^2 d\sigma_{s}]\\
& \geq  - 2[\frac{1}{\rho  r_0 } +\frac{\rho (1+\rho) }{3(1-\rho)} r_0 \lambda_1 ] \exp ( {\frac{2\sin(\epsilon r_0 ) s}{\sin [(1+\epsilon) r_0] \sin [(1-\epsilon )r_0] } })  \int_{\Sigma} f^2 d\sigma.
\end{aligned}
$$
Since $$\frac{\sin(\epsilon r_0 ) \epsilon r_0}{\sin [(1+\epsilon) r_0] \sin [(1-\epsilon )r_0] } 
=\frac{\epsilon}{1+\epsilon} \frac{\sin(\epsilon r_0 )}{\sin [(1-\epsilon )r_0] }\frac{(1+\epsilon) r_0}{\sin [(1+\epsilon) r_0] },$$
and $\frac{x}{\sin x}$ is increasing on $[0,\frac{\pi}{2}]$,$\frac{\sin (\epsilon x) }{\sin(1-\epsilon)x}\leq 1$ for any
$x \in [0,\frac{\pi}{4}] $ and any  $\epsilon \in [0,\frac{1}{2}]$, $$ \frac{\sin(\epsilon r_0 ) \epsilon r_0}{\sin [(1+\epsilon) r_0] \sin [(1-\epsilon )r_0] } \leq \frac{\pi}{6} ,\forall \epsilon \in [0,\frac{1}{2}].$$
Hence we  can choose small enough $\epsilon$  ( does not depend  on $r_0)$ when we fix $\rho$  and notice $1<\lambda_1\leq 2$) such that 
$$2\epsilon  [\frac{1}{\rho   } +\frac{\rho (1+\rho) }{3(1-\rho)} r_0 ^2 \lambda_1 ]  \exp ( {\frac{2\sin(\epsilon r_0 ) \epsilon r_0}{\sin (1+\epsilon) r_0 \sin (1-\epsilon )r_0 } }) <1 ,$$ 
and we have 
$$\int_{l=  \epsilon r_0 } (u-\phi (\epsilon r_0) )^2 d\sigma_{\epsilon r_0} >A
\int_{\Sigma} f^2 d\sigma, $$
where  $0<A<\frac{1}{2}$ is a small  positive constant independent of $r_0$. Hence by \eqref{eqn:lambda-mu}, when we also fix $\epsilon'>\epsilon$, we have 
$$\lambda_1 \geq 1+B {r_0}^6,$$
where $ B>0 $ is a small  constant independent of $r_0$. 	Since by corollary 1, we know $r_0$ has a lower bound depending only on the genus $g$, we complete the proof of the main theorem.
\end{proof}

\end{document}